\documentclass[12pt]{article}
\usepackage{amsmath,amssymb, amsfonts, amsthm,amscd}
\usepackage[T2A]{fontenc}
\usepackage[cp1251]{inputenc}
\usepackage[russian,ukrainian,english]{babel}
\usepackage{graphicx}

\newtheorem{proposition}{Proposition}

\newtheorem{theorem}{Theorem}

\newtheorem*{proposition*}{Proposition}
\newtheorem*{lemma*}{Lemma}
\newtheorem*{theorem*}{Theorem}
\newtheorem*{corollary*}{Corollary}

\newtheorem{definition}{Definition}

\newtheorem*{definition*}{Definition}
\newtheorem*{remark*}{Remark}
\newtheorem*{example*}{Example}

\begin{document}
\begin{center}
\textbf{Conditions for stable range of an elementary divisor rings}
\end{center}
\vskip 0.1cm \centerline{\textbf{Zabavsky Bohdan}}

\vskip 0.3cm

\centerline{\footnotesize{Department of Algebra and Logic,  Ivan Franko National University of Lviv,  Ukraine}}
\vskip 0.5cm

\centerline{\footnotesize{August, 2015}}
\vskip 0.7cm

\footnotesize{\noindent\textbf{Abstract:} \textit{Using the concept of ring of Gelfand range 1 we proved that a commutative Bezout domain is an elementary divisor ring iff it is a
ring of Gelfand range 1. Obtained results give a solution of problem
of elementary divisor rings for different classes of commutative
Bezout domains, in particular, $PM^{*}$, local Gelfand domains and
so on.}

}

\vskip 1cm

\normalsize

The problem of diagonalization of matrices is a classic one.
Definition of an elementary divisor ring was given by I. Kaplansky
in 1949 \cite{Kaplansky}. Since any elementary divisor ring is a
Bezout ring, we obtain the question: whether an arbitrary Bezout
ring is an elementary divisor ring \cite{{Henr},{Kaplansky},{LLS}}.
Gilman and Henriksen \cite{Gillman} constructed an example of a
commutative Bezout ring which is not an elementary divisor ring. The
question: whether any commutative Bezout domain is an elementary
divisor ring is still open.

All rings considered will be commutative and have identity. A ring
is a Bezout ring if every finitely generated ideal is principal. A
commutative ring is called a $PM$ -- ring if any its prime ideal is
contained in a unique maximal one. If the latter condition holds
only for nonzero prime ideals then a ring is called a $PM^{*}$ --
ring \cite{McG}.

A commutative ring $R$ is called a Gelfand ring if for any $a,b\in
R$ such that $a+b=1$ there exist elements $x,y\in R$ such that
$(1+ax)(1+by)=0$. In the case of commutative rings the class of
Gelfand rings is the same as the class of $PM$ -- rings \cite{McG}.

1. Gelfand range 1.

\begin{definition}\label{d1}
An element $a$ of a commutative ring is called a Gelfand element if
for any elements $b$, $c$ such that $bR+cR+aR=R$, we have $a=rs$
where $rR+bR=R$, $sR+cR=R$.
\end{definition}
\begin{theorem}\label{t1}
 Let $R$ be a commutative Bezout domain. An element $a$ is a Gelfand
 element iff the factor -- ring $R/aR$ is a $PM$ -- ring.
\end{theorem}
\begin{proof}
Let $\overline{R}=R/aR$ and let $a$ be a Gelfand element. We are
going to prove that $\overline{R}$ is a $PM$ -- ring. Let
$\overline{b}=b+aR$ and $\overline{c}=c+aR$ be any elements in
$\overline{R}$ and $bR+cR+aR=R$. Let
$\overline{b}\overline{R}+\overline{c}\overline{R}=\overline{R}$ and
by the definition of $a$, we have: $a=rs$, where $rR+bR=R$,
$sR+cR=R$. Let $\overline{r}=r+aR$, $\overline{s}=s+aR$, then
$\overline{0}=\overline{r}\overline{s}$,
$\overline{r}\overline{R}=\overline{b}\overline{R}=\overline{R}$ and
$\overline{s}\overline{R}+\overline{c}\overline{R}=\overline{R}$, i.
e. $\overline{R}$ is a Gelfand ring and hence $\overline{R}$ is a
$PM$ -- ring.

If $\overline{R}$ is a $PM$ -- ring then $\overline{R}$ is a Gelfand
ring, i. e. for any $\overline{b}, \overline{c}\in\overline{R}$ such
that
$\overline{b}\overline{R}+\overline{c}\overline{R}=\overline{R}$ we
have $\overline{0}=\overline{r}\overline{s}$, where
$\overline{r}\overline{R}+\overline{b}\overline{R}=\overline{R}$ and
$\overline{s}\overline{R}+\overline{c}\overline{R}=\overline{R}$.
Since $\overline{0}=0+aR=\overline{r}\overline{s}$, we have $rs\in
aR$, where $\overline{r}=r+aR$, $\overline{s}=s+aR$. Let
$rR+aR=r_{1}R$. Then $r=r_{1}r_{0}$, $a=r_{1}a_{0}$, where
$r_{0}R+a_{0}R=R$ \cite{Zabavsk}. Since $r_{0}R+a_{0}R=R$, we have
$r_{0}u+a_{0}v=1$ for some elements $u,v\in R$. Since $rs\in aR$, we
have $rs=at$ for some element $t\in R$. Then
$r_{1}r_{0}s=r_{1}a_{0}t$. Since $R$ is a domain, we have
$r_{0}s=a_{0}t$. Since $r_{0}u+a_{0}v=1$, we have
$a_{0}(tu+a_{0}v)=s$. Therefore, $a=r_{1}a_{1}$ and, since
$\overline{r}\overline{R}+\overline{b}\overline{R}=\overline{R}$, we
have
$\overline{r_{1}}\overline{R}+\overline{b}\overline{R}=\overline{R}$.
Then
$$R=r_{1}R+bR+aR=r_{1}R+bR+r_{1}r_{0}sR$$
and we have $r_{1}R+bR=R$. Since $a_{0}(tu+a_{0}v)=s$ and
$\overline{s}\overline{R}+\overline{c}\overline{R}=\overline{R}$
then
$\overline{a_{0}}\overline{R}+\overline{c}\overline{R}=\overline{R}$.
Then
$$R=a_{0}R+cR+aR=a_{0}R+cR+a_{0}r_{1}R=a_{0}R+cR.$$
\end{proof}

We have an obvious result

\begin{proposition}\label{p1}
 For a commutative Bezout domain $R$ the following conditions are
 equivalents:

 1) $a$ is a Gelfand element;

 2) for any prime ideal $P$ of $R$ such that $a\in P$ there exists a
 unique maximal ideal $M$ such that $P\in M$.
\end{proposition}
\begin{proof}
This is obvious, since $\overline{P}$ is a prime ideal of $R/aR$ iff
there exists a prime ideal $P$ of $R$ such that $aR\subset P$ and
$\overline{P}=P/aR$.
\end{proof}

\begin{definition}\label{d2}
Let $R$ be a commutative Bezout domain. We say that $R$ is a ring of
Gelfand range 1 if for any $a,b\in R$ such that $aR+bR=R$ there
exists an element $t\in R$ such that $a+bt$ is a Gelfand element.
\end{definition}

Since any unit element is a Gelfand element, we obtain the following
result.

\begin{proposition}\label{p3}
 A commutative Bezout domain of stable range 1 is a ring of Gelfand
 range 1.
\end{proposition}

Recall that a ring $R$ is said to be a ring of stable range 1 if for
any $a,b\in R$ such that $aR+bR=R$ there exists an element $t\in R$
such that $(a+bt)R=R$ \cite{Zabavsk}.

\begin{definition}\cite{{Zab21},{Zab22}}\label{d3}
Let $R$ be a commutative Bezout domain. An element $a\in R$ is said
to be an avoidable element if for any elements $c,b\in R$ such that
$aR+bR+cR=R$, we have $a=rs$ where $rR+bR=R$, $sR+cR=R$ and
$rR+sR=R$.
\end{definition}

Since any avoidable element is obviously a Gelfand element and by
\cite{{Zab21},{Zab22}} we have the following result.

\begin{theorem}\label{t2}
An elementary divisor domain is a ring of Gelfand range 1.
\end{theorem}

Furthemore we have the following result.

\begin{theorem}\label{t3}
Let $R$ be a commutative Bezout domain of Gelfand range 1. Then $R$
is an elementary divisor ring.
\end{theorem}
\begin{proof}
Let $ A=\begin{pmatrix}
a & 0 \\
b &c
\end{pmatrix}
$ with $aR+bR+cR=R$. By \cite{{Kaplansky},{Zabavsk}} for the proof
of the theorem it is enough to prove that $A$ admits a diagonal
reduction. Write $ax+by+cz=1$ for some elements $x,t,z\in R$. Then
$bR+(ax+cz)R=R$. Since $R$ is a Gelfand range 1, we see that there
exists some element $t\in R$ such that $d=b+(ax+cy)t$ is a Gelfand
element. We have
$$
\begin{pmatrix}
1 & 0 \\
xt &1
\end{pmatrix}
\begin{pmatrix}
a & 0 \\
b &c
\end{pmatrix}\begin{pmatrix}
1 & 0 \\
zt &1
\end{pmatrix}=\begin{pmatrix}
a & 0 \\
d &c
\end{pmatrix}
$$
where, obviously, $aR+dR+cR=R$ and $d$ is a Gelfand element.
According to the restrictions on $d$ and by Proposition \ref{p2} we
have $d=rs$ where $rR+aR=R$, $sR+cR=R$. Let $p\in R$ be an element
such that $sp+ck=1$ for some element $k\in R$. Hence $rsp+rck=r$ and
$dp+crk=r$. Denoting $rk=q$, we obtain $(dp+cq)R+aR=R$. Let
$pR+qR=\delta R$, then $p=p_{1}\delta$, $q=q_{1}\delta$ and
$\delta=pu+qv$ with $p_{1}R+q_{1}R=R$. Hence $pR\subset p_{1}R$ and
$pR+cR=R$, we have $p_{1}R+cR=R$. Since $p_{1}R+q_{1}R=R$ we have
$p_{1}R+(p_{1}d+q_{1}c)R=R$. Since $dp+cq=\delta(dp_{1}+cq_{1})$ and
$(dp+cq)R+aR=R$ we obtain $(dp_{1}+cq_{1})R+aR=R$. As well as
$p_{1}R+(p_{1}d+q_{1}c)R=R$, finally we have
$ap_{1}R+(dp_{1}+cq_{1})R=R$. By \cite{Kaplansky}, the matrix
$
\begin{pmatrix}
a & 0 \\
d &c
\end{pmatrix}
$ admits a diagonal reduction. Then, obviously, the matrix $A$
admits a diagonal reduction.
\end{proof}

Since by Theorem \ref{t1} and Proposition \ref{p1} a commutative
Bezout $PM^{*}$ -- domain is a domain of Gelfand range 1, as a
consequence of Theorem \ref{t3} we have the following result which
is answer to open questions raised in \cite{{LLS},{ZabMatSud}}.

\begin{theorem}\cite{Luhansk}\label{t4}
A commutative Bezout domain in which any nonzero prime ideal
contained in a unique maximal ideal is an elementary divisor domain.
\end{theorem}

\begin{definition}\label{d4}
A commutative Bezout ring is a local Gelfand ring if for any
elements $a,b\in R$ such that $aR+bR=R$ at least one of the elements
$a$, $b$ is a Gelfand element.
\end{definition}

An obvious example \cite{Luhansk} of a local Gelfand ring is the
Henriksen example \cite{Henr}
$$R=\{z_0 + a_1 x + a_2 x^2+...+a_n
x^{n}+...| z_0 \in \mathbb{Z}, a_i \in \mathbb{Q}\}.$$

Note that the Henriksen example is not a commutative Bezout domain
in which any nonzero prime ideal is contained in a unique maximal
ideal \cite{Henr}.

\begin{proposition}\label{p4}
 A local Gelfand ring  is a ring of Gelfand
 range 1.
\end{proposition}
\begin{proof}
Let $aR+bR=R$. There are two cases:

1) $a$ is a Gelfand element. Then $a+b0$ is a Gelfand element.

2) $a$ is not a Gelfand element. Since $aR+(a+b)R=R$ and $a$ is not
a Gelfand element, we see that $a+b1$ is a Gelfand element.
Therefore, $R$ is Gelfand range 1.
\end{proof}

By Proposition \ref{p4} and Theorem \ref{t3} we obtain the following
result.

\begin{theorem}\cite{Carp}\label{t5}
A local Gelfand ring is an elementary divisor ring.
\end{theorem}

2. Element of almost stable range 1.

\begin{definition}\label{d5}
Let $R$ be a commutative Bezout ring. An element $a$ is said to be
an element of almost stable range 1 if the factor -- ring $R/aR$ is
a ring of stable range 1.
\end{definition}

Set $S(R)=\{a| a \text{ is an element of almost stable range 1}\}$.
Obviously, $S(R)\neq\emptyset$, because $1\in S(R)$.

\begin{proposition}\cite{Shubani}\label{p5}
The $S(R)$ is saturated and multiplicatively closed.
\end{proposition}
\begin{proof}
Let $a,b\in S$. By \cite{HC01}, stable range of $R/(aR\cap bR)$ is
equal to 1. Since $$R/(aR\cap bR)\cong R/(ab)R/(aR\cap bR)/(ab)R,$$
we see that stable range of $R/(ab)R/(aR\cap bR)/(ab)R$ is equal to
1. Obviously, $(aR\cap bR)/(ab)R\subset J(R/(ab)R)$. This implies
that stable range of $R/(ab)R$ is equal to 1. Hence $ab\in S(R)$.

Assume that $a=bc$ with $a\in S$. Since $aR\subset bR$ we see that
$$
R/bR\cong R/aR/bR/aR
$$
whence stable range of $R/bR$ is equal 1. Thus $b\in S(R)$.
\end{proof}

Let $R$ be a commutative Bezout domain. Since $S(R)$ is a saturated
multiplicatively closed set, we can consider a localization of $R$
with respect to the set $S$, i. e. the ring of fractions $R_S$.

\begin{theorem}\label{t6}
A commutative Bezout domain is an elementary divisor ring if and
only if $R_S$ is an elementary divisor ring.
\end{theorem}
\begin{proof}
By \cite{{Kaplansky},{LLS}} it is sufficient to prove the statement
for any matrix $ A=\begin{pmatrix}
a & 0 \\
b &c
\end{pmatrix}
$, where $aR+bR+cR=R$. If $R_S$ is an elementary divisor ring then
for the elements $a,b,c\in R_S\cap R$, one can find elements
$ps^{-1}, qs^{-1}\in R_S$, $s\in S$, such that
$$
aps^{-1}R_S+(bps^{-1}+cqs^{-1})R_S=R_S.
$$
Then one can find elements $r,k\in R$ such that $apr+(bp+cq)k\in S$.
Since $S$ is a saturated set, we see that the matrix $A$ has an
equivalent matrix of the form $ B=\begin{pmatrix}
z & 0 \\
x &y
\end{pmatrix}
$, where $xR+yR+zR=R$ and $z\in S$. Since $z\in S$, by \cite{MG07}
we have $zR+(x+y\lambda)R=R$ for some element $\lambda\in R$. By
\cite{Kaplansky}, matrix $B$ and hence the matrix $A$ admits a
diagonal reduction. Therefore, $R$ is an elementary divisor ring.

Conversely, assume that $R$ is an elementary divisor ring. Assume
that we have arbitrary elements $as^{-1}$, $bs^{-1}$ and $cs^{-1}$
from $R_S$ such that $ as^{-1}R_S+(bs^{-1}+cs^{-1})R_S=R_S$. Then
$aR+bR+cR=dR$ for some element $d\in S(R)$. Let $a=a_{1}d$,
$b=b_{1}d$, $c=c_{1}d$ for some elements $a_{1},b_{1}, c_{1}\in R$
such that $a_{1}R+b_{1}R+c_{1}R=R$. Since $R$ is an elementary
divisor ring, by \cite{Kaplansky} there exist elements $p,q\in R$
such that $a_1pR+(b_1p+c_1q)R=R$. Then
$$
aps^{-1}R_S+(bps^{-1}+cqs^{-1})R_S =R_S.$$ By \cite{Kaplansky}, $R$
is an elementary divisor ring.
\end{proof}

Let $R$ be a commutative Bezout domain and let $S(R)$ be the set of
all its elements of almost stable range 1. Since $S(R)$ is a
saturated multiplicatively closed set, by applying the transfinite
induction, we construct a chain $\{R^{\alpha} |  \alpha \text{ is
ordinal }\}$ of saturated multiplicatively closed sets in the domain
$R$ as follows:

We set $R^0=S(R)$. Let $\alpha$ be an ordinal greater than zero and
assume that $R^{\beta}$ are already constructed and is a saturated
multiplicatively closed sets in $R$ for each $\beta$ such that
$\beta< \alpha$ and $K_{\beta}=R_{R^{\beta}}$. Obviously $K_{\beta}$
is a commutative Bezout domain and let $S(K_{\beta})$ be the set all
elements of almost stable range 1 of $K_{\beta}$. By Proposition
\ref{p5}, $S(K_{\beta})$ is a saturated multiplicatively closed set.
We define $R^{\alpha}$ as
$$R^{\alpha}=\bigcup\limits_{\beta< \alpha}R^{\beta},$$
if $\alpha$ is a limit ordinal, and

$$R^{\alpha}=S(K_{\alpha-1})\bigcap R$$
if $\alpha$ is not a limit ordinal. It is clear that $R^{\alpha}$ is
a saturated multiplicatively closed set and if $\alpha$ and $\beta$
are ordinals such that $\alpha \leq \beta$, then $R^{\alpha}\subset
R^{\beta}\subset R$. In addition, $R^{\alpha}=R^{\alpha+1}$ for some
ordinal $\alpha$. In the case when $R^{\alpha}\neq R^{\alpha+1}$ for
each ordinal $\alpha$, we have $card(R^{\alpha})>card(\alpha)$. If
we choose $\beta$ such that $card(\beta)>card(R)$, then we obtain
$card(\beta)>card(R)>card(R^{\beta})$, which is a contradiction. Now
let $\alpha_{0}$ be the least ordinal with the property
$\{R^{\alpha}  |  0\leq\alpha\leq\alpha_{0}\}$ is a $D$ -- chain in
$R$.

By analyzing Theorem \ref{t6} and the $D$ -- chain of the
commutative Bezout domain, we give an answer the question whether
the fact that a commutative Bezout domain is an elementary divisor
ring is equivalent to the case of a commutative Bezout domain with
trivial (unit) elements of almost stable range 1, i. e. units and
only units are elements of almost stable range 1. Let $R$ be a
commutative Bezout domain in which any element of almost stable
range 1 is a unit, i. e. $U(R)=S$. We formulate the following
question: when $R$ is an elementary divisor ring? By \cite{Zab22},
$R$ is an elementary divisor ring iff for any $a,b\in R$ such that
$aR+bR=R$ there exist an element $t\in R$ such that $a+bt$ is an
avoidable element, i. e. $R/(a+bt)R$ is a clean ring. Since every
clean ring is a ring of idempotent stable range 1, $a+bt$ is an
element of almost stable range 1. Since $U(R)=S(R)$, we see that
$a+bt$ is unit element, i. e. $R$ is a ring of stable range 1. Since
any element of stable range 1 is an element of almost stable range
1, we obtain the following result.

\begin{theorem}\label{t7}
Let $R$ be a commutative Bezout domain such that $U(R)=S(R)$. Then
$R$ is an elementary divisor ring if and only if $R$ is a field.
\end{theorem}

Hence the problem "is every commutative Bezout domain an elementary
divisor ring?" is equivalent to the problem does every commutative
Bezout domain contain a non -- unit element of almost stable range
1.

\end{document}